\documentclass[12pt]{amsart}
\usepackage{a4wide}
\usepackage[T1]{fontenc}
\usepackage{amssymb,amsmath,amsthm,latexsym}
\usepackage{mathrsfs}
\usepackage[usenames,dvipsnames]{color}
\usepackage{euscript}
\usepackage{graphicx}
\usepackage{mdwlist}
\usepackage{enumerate}
\usepackage{xcolor}
\usepackage{mathtools,dsfont,wasysym}
\usepackage{stmaryrd}
\usepackage{ulem}
\usepackage{hyperref}
\usepackage{enumitem}
\hypersetup{colorlinks=true, linkcolor=blue, citecolor=red, urlcolor=cyan}

\newtheorem{theorem}{Theorem}[section]
\newtheorem{lemma}[theorem]{Lemma}
\newtheorem{corollary}[theorem]{Corollary}
\newtheorem{proposition}[theorem]{Proposition}

\newcounter{maintheorem}

\theoremstyle{remark}

\theoremstyle{definition}

\numberwithin{equation}{section}
\makeatother

\newcommand{\R}{\mathbb{R}}
\newcommand{\N}{\mathbb{N}}
\newcommand{\e}{\varepsilon}

\newcommand{\nn}[1]{{\left\vert\kern-0.25ex\left\vert\kern-0.25ex\left\vert #1 \right\vert\kern-0.25ex \right\vert\kern-0.25ex \right\vert}}

\renewcommand{\leq}{\leqslant}
\renewcommand{\geq}{\geqslant}

\DeclareMathOperator{\diam}{diam}

\DeclareMathOperator{\spn}{span}

\newcounter{smallromans}

{\end{list}}

\newcommand{\sna}{\operatorname{SNA}}

\newcommand{\lip}{\operatorname{Lip}_0}
\newcommand{\Lip}{\operatorname{Lip}}
\newcommand{\card}{\operatorname{card}}
\newcommand{\dens}{\operatorname{dens}}
\newcommand{\cof}{\operatorname{cof}}
\makeatletter
\renewcommand{\tocsection}[3]{%
	\indentlabel{\@ifnotempty{#2}{\bfseries\ignorespaces#1 #2\quad}}\bfseries#3}
\renewcommand{\tocsubsection}[3]{%
	\indentlabel{\@ifnotempty{#2}{\ignorespaces#1 #2\quad}}#3}
\newcommand\@dotsep{4.5}
\def\@tocline#1#2#3#4#5#6#7{\relax
	\ifnum #1>\c@tocdepth % then omit
	\else
	\par \addpenalty\@secpenalty\addvspace{#2}%
	\begingroup \hyphenpenalty\@M
	\@ifempty{#4}{%
		\@tempdima\csname r@tocindent\number#1\endcsname\relax
	}{%
		\@tempdima#4\relax
	}%
	\parindent\z@ \leftskip#3\relax \advance\leftskip\@tempdima\relax
	\rightskip\@pnumwidth plus1em \parfillskip-\@pnumwidth
	#5\leavevmode\hskip-\@tempdima{#6}\nobreak
	\leaders\hbox{$\m@th\mkern \@dotsep mu\hbox{.}\mkern \@dotsep mu$}\hfill
	\nobreak
	\hbox to\@pnumwidth{\@tocpagenum{\ifnum#1=1\bfseries\fi#7}}\par% <-- \bfseries for \section page
	\nobreak
	\endgroup
	\fi}
\AtBeginDocument{%
	\expandafter\renewcommand\csname r@tocindent0\endcsname{0pt}
}
\def\l@subsection{\@tocline{2}{0pt}{2.5pc}{5pc}{}}
\makeatother

\makeatletter

\begin{document}
\title[On isometric embeddings into the set SNA]{On isometric embeddings into the set of strongly norm-attaining Lipschitz functions}

%On the spaceability of the set of strongly norm-attaining Lipschitz functions

\author[S.~Dantas]{Sheldon Dantas}
\address[S.~Dantas]{Departament de Matem\`atiques and Institut Universitari de Matem\`atiques i Aplicacions de Castell\'o (IMAC), Universitat Jaume I, Campus del Riu Sec. s/n, 12071 Castell\'o, Spain \href{https://orcid.org/0000-0001-8117-3760}{ORCID: \texttt{0000-0001-8117-3760}}}
\email{\texttt{dantas@uji.es}}

\author[R. Medina]{Rubén Medina}
\address[R. Medina]{Universidad de Granada, Facultad de Ciencias. Departamento de Análisis Matemático, 18071-Granada (Spain); and Czech Technical University in Prague, Faculty of Electrical Engineering. Department of Mathematics, Technická 2, 166 27 Praha 6 (Czech Republic) %\newline
\href{https://orcid.org/0000-0002-4925-0057}{ORCID: \texttt{0000-0002-4925-0057}}}
\email{rubenmedina@ugr.es}

\author[A.\ Quilis]{Andr\'es Quilis}
\address[Andr\'es Quilis]{Instituto Universitario de Matemática Pura y Aplicada, Universitat Politècnica de València, Camí de Vera, s/n; and Czech Technical University in Prague, Faculty of Electrical Engineering. Department of Mathematics, Technická 2, 166 27 Praha 6 (Czech Republic) %\newline 
\href{https://orcid.org/0000-0001-6022-9286}{ORCID: \texttt{0000-0001-6022-9286} }}
\email{anquisan@posgrado.upv.es}

\author[\'O.\ Rold\'an]{\'Oscar Rold\'an}
\address[\'Oscar Rold\'an]{Departamento de An\'{a}lisis Matem\'{a}tico,
Universidad de Valencia, Doctor Moliner 50, 46100 Burjasot (Valencia), Spain.
\href{https://orcid.org/0000-0002-1966-1330}{ORCID: \texttt{0000-0002-1966-1330} }}
\email{oscar.roldan@uv.es}

\thanks{}

% \subjclass[2020]{Primary: 46B04; Secondary: 46B20, 46B87}
% \date{\today}
% \keywords{Norm-attaining Lipschitz functionals, Lipschitz-free Banach space, ...}
\date{\today}

\begin{abstract} In this paper, we provide an infinite metric space $M$ such that the set $\sna(M)$ of strongly norm-attaining Lipschitz functions does not contain a subspace which is isometric to $c_0$. This answers a question posed by Antonio Avilés, Gonzalo Martínez Cervantes, Abraham Rueda Zoca, and Pedro Tradacete. On the other hand, we prove that $\sna(M)$ contains an isometric copy of $c_0$ whenever $M$ is a metric space which is not uniformly discrete. In particular, the latter holds true for infinite compact metric spaces while it does not for proper metric spaces. Some positive results in the non-separable setting are also given. 
\end{abstract}
\maketitle
\tableofcontents

%-------------------------------------------------------%
%                                                     %
% 						INTRODUCTION 					%
%                                                     %
%-------------------------------------------------------%

\section{Introduction}

Let $M$ be a metric space. We consider the subset $\sna(M)$ of $\lip(M)$ of all strongly norm-attaining Lipschitz functions on $M$. In 2016, Marek Cúth, Michal Doucha, and Przemysław Wojtaszczyk \cite{CDW16} proved that $\ell_{\infty}$ (and hence $c_0$) embeds isomorphically in $\lip(M)$ for any infinite metric space $M$. One year later, this result was improved by Marek Cúth and Michal Johanis and it is known now that $\ell_{\infty}$ (and hence $c_0$) embeds isometrically in $\lip(M)$ \cite{CJ17}. Motivated by the papers \cite{AMRTpre, Godefroy01, KRpre}, we turn our attention to the study of the analogous problems for the subset $\sna(M)$. Drastically different from the classical norm-attaining theory, where Martin Rmoutil \cite{Rmoutil17} proved that the set of all norm-attaining functionals needs not contain $2$-dimensional spaces, Antonio Avilés, Gonzalo Martínez Cervantes, Abraham Rueda Zoca, and Pedro Tradacete \cite{AMRTpre} provided a beautiful and interesting construction, and showed that $\sna(M)$ contains an isomorphic copy of $c_0$ for every infinite metric space $M$ (answering \cite[Question 2]{KRpre}). At the very ending of that paper, the authors wondered whether an isometric version of this result holds true (see \cite[Remark 3.6]{AMRTpre}).

In this paper, we answer the latter question by providing an example of an infinite uniformly discrete metric space $M$ so that $\sna(M)$ does {\it not} contain any subspace which is isometric to $c_0$ (see Theorem \ref{main-theorem}). On the other hand, we prove that $\sna(M)$ does contain an isometric copy of $c_0$ whenever $M$ is infinite but not uniformly discrete (see Theorem \ref{caso-compacto}). It turns out that this is no longer true even for proper metric spaces (see Theorem \ref{ContraejemploProper}). We conclude the paper by tackling the problem in the non-separable setting; we prove that whenever $\dens(M') = \Gamma > \omega_0$, then $\sna(M)$ contains an isometric copy of $c_0(\Gamma)$ (see Theorem \ref{gamma}).

\section{Preliminaries and notation}

Throughout the text, all the vector spaces will be considered to be {\it real}. Let $(M,d)$ be a pointed metric space (that is, a metric space with a distinguished point $0$). We denote by $\lip(M)$ the Banach space of all Lipschitz functions $f:M\rightarrow \mathbb{R}$ such that $f(0)=0$ endowed with the Lipschitz norm
\begin{equation*} 
\|f\|_{\Lip}:=\sup\left\{ \frac{|f(y)-f(x)|}{d(x,y)}\colon x,y\in M,\, x\neq y \right\}.
\end{equation*} 
We say that a Lipschitz function $f\in\lip(M)$ \textit{strongly attains its norm}, or that it is {\it strongly norm-attaining}, if there exist two different points $p,q\in M$ such that $$\|f\|_{\Lip} = \frac{|f(p)-f(q)|}{d(p,q)}.$$
The set of strongly norm-attaining Lipschitz functions on $M$ will be denoted by $\sna(M)$. In the past few years, this topic has been intensively studied. We send the reader to \cite{AMRTpre, CCGMR19, Chiclana, CGMR21, CM19, CM2022, CCM20, GPPR18, GPR17, Godefroy15, Godefroy16, JMRpre, KMS16, KRpre} and the references therein. 

\vspace{0.2cm} 

In what concerns this paper, without loss of generality, we will be reducing ourselves to the following contexts:
\begin{itemize}
    \item[(i)] All the metric spaces will be considered to be complete and infinite.
    \item[(ii)] The choice of the distinguised point $0$ in $M$ is irrelevant.
\end{itemize}
The reason we can consider (i) is that, by continuity, Lipschitz mappings have a unique norm-preserving extension to the completion of $M$. For (ii), let us notice that since the mapping $f\mapsto f-f(0)$ is a linear isometry in $\lip(M)$ that completely preserves the strong norm-attainment behaviour of the mappings, we do not need to worry about the choice of the distinguished point. Clearly, the problem of the isometric containment of $c_0$ in $\sna(M)$ is restricted to infinite metric spaces $M$.

\vspace{0.2cm} 

For us, the expression linear subspaces of $\sna(M)$ should be understood as linear subspaces of $\lip(M)$ consisting of strongly norm-attaining Lipschitz functions. Also, if $Y$ is a Banach space, then we say that $Y$ isometrically embeds in $\sna(M)$ (or, equivalently, $\sna(M)$ contains an isometric copy of $Y$), whenever there exists a linear isometric embedding $U: Y \rightarrow \lip(M)$ such that $U(Y) \subseteq \sna(M)$.

\vspace{0.2cm} 

The \textit{separation radius} of a point $x \in M$ is defined by 
\begin{equation*} 
R(x):=\inf \Big\{ d(x,y)\colon y\in M\setminus\{x\} \Big\},
\end{equation*} 
which will be central in some of the upcoming results. We will say that a point $x$ from a metric space $M$ \textit{attains its separation radius} whenever there is $y\in M$ such that $R(x)=d(x,y)$.
The symbol $M'$ stands for the set of all cluster points of $M$. Recall that a metric space $M$ is said to be \textit{discrete} if $M'=\emptyset$, \textit{uniformly discrete} if $\inf\{R(x)\colon x\in M\}>0$, and \textit{proper} if every closed and bounded subset of $M$ is compact. The notation $B(x,R)$ stands for the closed ball of center $x\in M$ and radius $R>0$.

\vspace{0.2cm}

Thorughout the entire note, we will denote by $c_0(\Gamma)$ the space of all real valued functions over a set $\Gamma$ satisfying that for every element $x\in c_0(\Gamma)$, the set $\{\gamma\in\Gamma\colon |x(\gamma)|\geq\varepsilon\}$ is finite for every $\e>0$. If $\Gamma$ is countable we will refer to $c_0(\Gamma)$ simply as $c_0$.

Let $X$ be a separable Banach space with a Schauder basis denoted by $\{x_n\}_{n=1}^\infty$. We say that a sequence $\{y_n\}_{n=1}^\infty$ in a Banach space $Y$ is \textit{(isometrically) equivalent} to the basis $\{x_n\}_{n=1}^\infty$ if there exists a linear (isometric) isomorphism $T\colon \overline{\text{span}}\{y_n\colon n\in\mathbb{N}\}\rightarrow X$ such that $T(y_n)=x_n$ for all $n\in\mathbb{N}$. The following straightforward facts will be used throughout the text without any explicit reference.
\begin{itemize}
    \item[(i)] A sequence $\{x_n\}_{n=1}^\infty$ is isometrically equivalent to the canonical basis of $c_0$ if and only if the equality $\big\|\sum_{n=1}^\infty \lambda_n x_n\big\|=\max_{n}|\lambda_n|$ holds for every sequence $\{\lambda_n\}_{n=1}^\infty \in c_0$.
    \item[(ii)] If a sequence $\{x_n\}_{n=1}^\infty$ is isometrically equivalent to the canonical basis of $c_0$, then so is the sequence $\{\varepsilon_nx_n\}_{n=1}^\infty$, where $\varepsilon_n\in\{-1,1\}$ for every $n\in\mathbb{N}$.
    \item[(iii)] Any subsequence of a sequence which is isometrically equivalent to the canonical basis of $c_0$ is once again isometrically equivalent to the same basis.
\end{itemize}

\vspace{0.2cm}

Given any set $A$ and a natural number $k\in\mathbb{N}$, we denote by $A^{[k]}$ the set of all subsets of $A$ with exactly $k$ elements. We will use Ramsey's Theorem intensively throughout the text, which ensures that given any infinite set $A$ and any finite partition of the set $A^{[k]}$, $\{B_1,\dots,B_n\}$ for some $n\in\mathbb{N}$, there exists an infinite subset $S$ of $A$ and a number $i\in\{1,\dots n\}$ such that $S^{[k]}$ is contained in $B_i$ (see, for instance, \cite[Proposition 6.4]{FHHMZ11}).

\section{Some useful tools}

In this section, we state and prove some auxiliary results that will be crucial for the rest of the note. The following are three essential yet straightforward lemmas that hold in any complete metric space. We provide their proofs for the sake of completeness.

\begin{lemma}
\label{lemma1}
Let $M$ be a complete metric space. Suppose that $\{f_n\}_{n=1}^\infty\subseteq \lip(M)$ is a sequence isometrically equivalent to the canonical basis of $c_0$. Then, for every $n\in\mathbb{N}$, if the function $f_n$ strongly attains its Lipschitz norm at a pair of points $x_n,y_n\in M$, then $f_m(x_n)=f_m(y_n)$ for every $m\in\mathbb{N}\setminus\{n\}$. 
\end{lemma}
\begin{proof}
    Let $n\in\mathbb{N}$ be fixed. Suppose that the function $f_n$ strongly attains its Lipschitz norm at a pair of points $x_n,y_n\in M$. Without loss of generality, we may (and we do) assume that $|f_n(x_n)-f_n(y_n)|=f_n(x_n)-f_n(y_n)=d(x_n,y_n)$. Let us suppose by contradiction that there exist natural numbers $m\neq n$ such that $f_m(x_n)\neq f_m(y_n)$. We may again suppose without loss of generality that $f_m(x_n)>f_m(y_n)$ (otherwise we may consider the sequence $\{g_k\}_{k=1}^\infty$ defined as $g_m=-f_m$ and $g_k=f_k$ for $k\neq m$, which is still equivalent to the $c_0$ basis). Set $f:=f_n+f_m$. Then, we have that 
    \begin{align*}
        |f(x_n)-f(y_n)|&\geq (f_n+f_m)(x_n)-(f_n+f_m)(y_n)\\
        &=f_n(x_n)-f_n(y_n)+f_m(x_n)-f_m(y_n)\\
        &>d(x_n,y_n),
    \end{align*}
    which yields a contradiction with the fact that $f$ is $1$-Lipschitz.
\end{proof}

% \begin{lemma}
% \label{lemma2}
% Let $M$ be a complete metric space. Suppose that $\{f_n\}_{n=1}^\infty\subseteq \lip(M)$ is a sequence equivalent to the $c_0$ canonical basis. Then, for every subsequence $\{f_{n_k}\}_{k\in\mathbb{N}}$ and every point $p\in M$ we have that $\lim_{k\rightarrow\infty}|f_{n_k}(p)|=0$. 
% \end{lemma}
% \begin{proof}
% Let $T\colon c_0\rightarrow \overline{\text{span}}\{f_n\colon n\in\mathbb{N}\}$ be a linear isomorphism with $T(e_n)=f_n$ for all $n\in \N$ and put $C= \|T\|$. Suppose that there exists a subsequence $\{f_{n_k}\}_{k\in\mathbb{N}}$ and a point $p\in M$ such that $\lim_{k\rightarrow\infty}|f_{n_k}(p)|> 0$. Then, there exists $N\in\mathbb{N}$ such that $\sum_{k=1}^N |f_{n_k}(p)|> C\cdot d(p,0)$. However, this implies that there exist $\{\varepsilon_k\}_{k=1}^N\subseteq\{-1,1\}^N$ such that the function $\sum_{k=1}^N \varepsilon_k f_{n_k} $ is not $C$-Lipschitz, contradicting the fact that the operator norm of $T$ is $C$.
% \end{proof}
\begin{lemma}
\label{lemma2}
Let $M$ be a complete metric space. Suppose that $\{f_n\}_{n=1}^\infty\subseteq \lip(M)$ is a sequence equivalent to the canonical basis of $c_0$. Then, for all $p\in M$, $\lim_{n\rightarrow\infty}|f_n(p)|=0$. 
\end{lemma}
\begin{proof}
Let $T\colon c_0\rightarrow \overline{\text{span}}\{f_n\colon n\in\mathbb{N}\}$ be a linear isomorphism with $T(e_n)=f_n$ for all $n\in \N$ and set $C=\|T\|$. Suppose that for some $p\in M$, the sequence $\{f_n(p)\}_{n=1}^{\infty}$ does not converge to $0$. Then, there exists $N\in\mathbb{N}$ such that $\sum_{n=1}^N |f_{n}(p)|> C\cdot d(p,0)$. However, this implies that there exist $\{\varepsilon_n\}_{n=1}^N\subseteq\{-1,1\}^N$ such that the function $\sum_{k=1}^N \varepsilon_n f_{n}$ is not $C$-Lipschitz, contradicting the fact that the operator norm of $T$ is $C$.
\end{proof}

The following lemma is an immediate consequence of the triangle inequality.

\begin{lemma} \label{fact} Let $f \in \lip(M)$ be given. Suppose that $x, y \in M$ with $x\not=y$ are such that $|f(x) - f(y)| = d(x,y)$. Then, we have that 
\begin{equation*}
    |f(x) - C| + |f(y) - C| \geq d(x,y)
\end{equation*}
for every $C \in \R$.
\end{lemma}

Finally, for the upcoming positive results of the paper, we need the following generalization of \cite[Lemma 3.1]{AMRTpre}.

\begin{lemma}\label{prop:tent-c02}
Let $\Gamma$ be a nonempty index set. Let $M$ be a pointed metric space such that there exist two sets $\{x_{\gamma}\}_{\gamma\in\Gamma},\{y_\gamma\}_{\gamma\in\Gamma}\subseteq M$ with $x_\gamma\neq y_\gamma$, $x_\alpha\neq x_\beta$ for $\gamma,\alpha,\beta\in\Gamma$, $\alpha\neq\beta$. If $d(x_\alpha, x_\beta)\geq d(x_\alpha,y_\alpha) + d(x_\beta,y_\beta)$ for every $\alpha\neq \beta\in \Gamma$,
then there is a linear subspace of $\sna(M)$ isometric to $c_0(\Gamma)$.
\end{lemma}
\begin{proof}
For each $\gamma\in \Gamma$, define $f_\gamma:M\rightarrow \mathbb{R}$ by
$$f_\gamma(x):= \max\{0, d(x_\gamma,y_\gamma) - d(x, x_\gamma)\} \ \ \ (x\in M).
$$
Let $\lambda :=\{\lambda_\gamma\colon \gamma\in\Gamma\}\in c_0(\Gamma)$ and let $\gamma_0\in\Gamma$ be such that $|\lambda_{\gamma_0}|=\| \lambda \|_\infty$. Finally, set $f:M\rightarrow\mathbb{R}$ to be defined as
$$f:=\sum_{\gamma\in\Gamma} \lambda_\gamma (f_\gamma-f_\gamma(0)).$$
We will be done when we check that $f$ is an element of $\lip(M)$ with Lipschitz norm $\|\lambda\|_{\infty}$ strongly attaining its norm at the pair $(x_{\gamma_0}, y_{\gamma_0})$.

% By Lemma \ref{lemma:tent1}, we can assume without loss of generality that $|\Gamma|>1$. Assume first that $\Gamma$ is at most countable.
It is easy to check that, for all $\gamma\in\Gamma$, $f_\gamma$ strongly attains its norm at the pair $(x_\gamma, y_\gamma)$ with $\|f_\gamma\|_{\Lip}=1$. Also, let us notice that the support of $f$ lies in $\bigcup_{\gamma\in\Gamma} B(x_\gamma, d(x_\gamma,y_\gamma))$. %Note also that since the choice of the point $0$ is irrelevant and since at most one of the functions $f_\gamma$ can have a non-zero value at $0$, 
Note that we can assume without loss of generality that $f_\gamma(0)=0$ for all $\gamma\in\Gamma$.

Let us now prove that $\|f\|_{\Lip}=|\lambda_{\gamma_0}|$. Let $x\neq y$ be two points in $M$. We will distinguish several cases. 
\begin{itemize}
\item[(a)] If both $x$ and $y$ lie outside of $\bigcup_{\gamma\in\Gamma} B(x_\gamma, d(x_\gamma,y_\gamma))$, then clearly $|f(x)-f(y)|=0$.
\vspace{0.2cm} 
\item[(b)] Assume that $x\notin \bigcup_{\gamma\in\Gamma} B(x_\gamma, \delta_\gamma)$ and that there exists some $\alpha\in\Gamma$ such that $y\in B(x_\alpha,d(x_\alpha,y_\alpha))$. Since $d(x,y)\geq d(x_\alpha,x)-d(x_\alpha,y)\geq d(x_\alpha,y_\alpha)-d(x_\alpha,y)\geq 0$, we have 
$$\frac{|f(x)-f(y)|}{d(x,y)}=\frac{|\lambda_\alpha|(d(x_\alpha,y_\alpha) - d(x_\alpha, y))}{d(x,y)}\leq |\lambda_{\gamma_0}|.$$
\item[(c)] Assume now that there is some $\gamma\in\Gamma$ such that $x,y\in B(x_\gamma,d(x_\gamma,y_\gamma))$. Then, since $d(x_\gamma,y)\leq d(x_\gamma,x)+d(x,y)$ and $d(x_\gamma,x)\leq d(x_\gamma,y)+d(x,y)$, we have
$$\frac{|f(x)-f(y)|}{d(x,y)}=|\lambda_\gamma|\frac{|(d(x_\gamma,y_\gamma)-d(x_\gamma, x))-(d(x_\gamma,y_\gamma)-d(x_\gamma, y))|}{d(x,y)}\leq |\lambda_{\gamma_0}|.$$
\vspace{0.2cm} 
\item[(d)] Finally, if there are different $\alpha,\beta\in\Gamma$ such that $x\in B(x_\alpha,d(x_\alpha,y_\alpha))$ and $y\in B(x_\beta, d(x_\beta,y_\beta))$, assuming without loss of generality that $|\lambda_\alpha|\geq |\lambda_\beta|>0$, we have
\begin{equation}\label{eqn:2asteriscos}
\frac{|f(x)-f(y)|}{d(x,y)}=\frac{|\lambda_\alpha(d(x_\alpha,y_\alpha) - d(x,x_\alpha)) - \lambda_\beta(d(x_\beta,y_\beta)-d(y,x_\beta))|}{d(x,y)}.
\end{equation}

We will distinguish 2 cases now.

\vspace{0.5cm} 
\noindent
\textit{Case 1}: $\lambda_\alpha$ and $\lambda_\beta$ have the same sign. 
\vspace{0.5cm} 

We can assume for the following computation that both are positive (if not we simply multiply by $(-1)$ if needed). For $\gamma\in\{\alpha,\beta\}$, define $g_\gamma:M\rightarrow \mathbb{R}$ by $g_\gamma(z):=\lambda_\gamma(d(x_\gamma,y_\gamma) - d(x_\gamma, z))$ for all $z\in M$. Notice that in (\ref{eqn:2asteriscos}) we can clearly assume that $f(x)\neq f(y)$. If $f(x)>f(y)$, since $f(y)\geq 0\geq g_{\alpha}(y)$, we get that
\begin{align*}
\frac{|f(x)-f(y)|}{d(x,y)}&\leq\frac{f(x)-g_\alpha(y)}{d(x,y)}\leq |\lambda_{\gamma_0}|.
\end{align*}
The case where $f(y)>f(x)$ can be solved similarly with a symmetric argument.

% $$(**)=\frac{f(x)-f(y)}{d(x,y)}\leq\frac{f(x)-g_\alpha(y)}{d(x,y)}=(***),$$
% since $f(y)\geq 0 \geq g_\alpha(y)$. By definition of $f$ and $g_\alpha$ we would then have
% $$(***)=a_\alpha\frac{d(x_\alpha,y_\alpha) - d(x,x_\alpha) - d(x_\alpha,y_\alpha) + d(x_\alpha,y)}{d(x,y)}\leq a_\alpha,$$
% where the last inequality comes from using the triangle inequality (indeed, just note that $d(x_\alpha,y)\leq d(x_\alpha,x)+d(x,y)$). The case where $f(y)>f(x)$ can be solved similarly with a symmetric argument.

% On the other hand, if $f(y)>f(x)$, we can do a symmetric argument:
% $$(**)=\frac{f(y)-f(x)}{d(x,y)}\leq\frac{f(y)-g_\beta(x)}{d(x,y)}=(***),$$
% since $f(x)\geq 0 \geq g_\beta(x)$. By definition of $f$ and $g_\beta$ we would then have
% $$(***)=a_\beta\frac{d(x_\beta,y_\beta) - d(y,x_\beta) - d(x_\beta,y_\beta) + d(x_\beta,x)}{d(x,y)}\leq a_\beta,$$
% where the last inequality comes from using the triangle inequality once more (since $d(x_\beta,x)\leq d(x_\beta,y)+d(x,y)$).
% \begin{figure}[H]
%     \centering
%     \includegraphics[width=\textwidth]{Tent2.png}
%     \caption{Case 1 at the left (green slope is smaller than red slope in absolute value). Case 2 at the right (green slope is smaller than red slope, which is smaller than blue slope, in absolute value).}
% \end{figure}

\vspace{0.5cm} 
\noindent
\textit{Case 2}: $\lambda_\alpha$ and $\lambda_\beta$ have different signs. 

\vspace{0.5cm} 

We can assume that $\lambda_\alpha$ is positive and $\lambda_\beta$ is negative. By the triangle inequality, it is easy to see that 
\begin{equation*}
    \lambda_\alpha(d(x_\alpha,y_\alpha) - d(x_\alpha,y)) \leq (-\lambda_\alpha)(d(x_\beta,y_\beta)-d(x_\beta,y)).
\end{equation*}
%Note now that $d(x_\alpha, y)+d(y, x_\beta)\geq d(x_\alpha, x_\beta)\geq d(x_\alpha,y_\alpha)+d(x_\beta,y_\beta)$, so in particular we have that $\lambda_\alpha(d(x_\alpha,y_\alpha)+d(x_\beta,y_\beta))\leq \lambda_\alpha(d(x_\alpha,y)+d(y,x_\beta))$, that is, $\lambda_\alpha(d(x_\alpha,y_\alpha) - d(x_\alpha,y)) \leq (-\lambda_\alpha)(d(x_\beta,y_\beta)-d(x_\beta,y))$. 
Hence, if $g_\alpha$ is defined as in Case 1, then we get once more that $g_\alpha(y) \leq f(y)$ since $f(y)\leq 0$. Now we can repeat the same computations as in Case 1.
%$$g_\alpha(y)\leq \frac{(-\lambda_\alpha)}{\lambda_\beta}f(y)\leq f(y)$$

% same argument as in the first subcase of Case 1:
% $$(***)\leq\frac{f(x)-g_\alpha(y)}{d(x,y)}=a_\alpha\frac{\delta_\alpha - d(x,x_\alpha) - d(x_\alpha,y_\alpha) + d(x_\alpha,y)}{d(x,y)}\leq a_\alpha,$$
% by the triangle inequality.
\end{itemize}

This proves that $\|f\|_{\Lip}\leq|\lambda_{\gamma_0}|$. Finally, it is clear that $f$ strongly attains its Lipschitz norm at the pair $(x_{\gamma_0},y_{\gamma_0})$ and we are done.
\end{proof}

%$$\frac{|f(x_{\gamma_0})-f(y_{\gamma_0})|}{d(x_{\gamma_0},y_{\gamma_0})}=|\lambda_{\gamma_0}|,$$
%and so, $f$ strongly attains its Lipschitz norm, $\|f\|_{\Lip}=|\lambda_{\gamma_0}|$, at the pair $(x_{\gamma_0},y_{\gamma_0})$.

\section{The isometric containment of $c_0$ in $\sna(M)$}

In this section we turn our attention to the main results of the paper.

\subsection{A bounded and uniformly discrete counterexample}

In this subsection, we construct an infinite complete metric space $M$ such that the set $\sna(M)$ of strongly norm-attaining Lipschitz functions does not contain an isometric copy of $c_0$, answering a question posed in \cite[Remark 3.6]{AMRTpre}. It is worth mentioning that no point of the constructed metric space attains its separation radius.

\begin{theorem} \label{main-theorem} There exists an infinite bounded uniformly discrete complete metric space $M$ such that $c_0$ is not isometrically contained in $\sna(M)$ and for which no point in $M$ attains its separation radius.
\end{theorem}

\begin{proof} 
Let $M=\{p_n\}_{n\in\mathbb{N}}$ be any countable set endowed with the metric $d$ given by 
$$d(p_n,p_m)=
\begin{cases}
1+\frac{1}{\max\{m,n\}}\;\;\;&\text{ if }m\neq n,\\
0&\text{ otherwise.}
\end{cases}$$ 
Note that the diameter of $M$ is $3/2$.

For the sake of contradiction, let us suppose that there exists a sequence $\{f_n\}_{n\in\N}$ of strongly norm-attaining functions which is isometrically equivalent to the canonical basis of $c_0$. For every $n \in \N$, let $x_n, y_n \in M$ be such that $x_n \not= y_n$ and $\big|f_n(x_n) - f_n(y_n)\big| = d(x_n, y_n)$. Our goal is to find two natural numbers $n_0 \not= m_0$ and $\delta \in \{-1,1\}$ such that the Lipschitz function $f_{n_0}+\delta f_{m_0}$ has Lipschitz norm strictly greater than 1. This will lead to a contradiction.

Let us consider the sets
\begin{align*}
A &:= \left\{\{n,m\} \in \N^{[2]}: \{x_n, y_n\} \cap \{x_m, y_m\} = \emptyset\right\},\\
B_1 &:= \left\{\{n,m\} \in \N^{[2]}: x_n = x_m\right\},\\
B_2 &:=\left\{\{n,m\} \in \N^{[2]}: y_n = y_m\right\}\text{, and}\\
B_3 &:=\left\{\{n,m\} \in \N^{[2]}: x_n = y_m\text{ or }x_m=y_n\right\}.
\end{align*}

By Ramsey's theorem, there exists $C \in \{A, B_1, B_2, B_3\}$ and an infinite set $S \subseteq \N$ such that $S^{[2]} \subseteq C$. 

\vspace{0.2cm} 
\noindent
\textit{Case 1}: $C = A$. 
\vspace{0.2cm} 

We may assume by passing to a subsequence that $\{x_n, y_n\} \cap \{x_m, y_m\} = \emptyset$ for every $n, m \in \N$ with $n\not=m$. For each $n \in \N$, let us set 
\begin{equation*}
    \e_n := \frac{1}{2 k(n)}, \ \ \mbox{where} \ \ k(n) := \max \{ k \in \N: p_k = x_n \ \ \mbox{or} \ \ p_k = y_n \}.
\end{equation*}
Let us fix $n_0 \in \N$. Since $\{x_n, y_n\} \cap \{x_m, y_m\} = \emptyset$ for every $n, m \in \N$ with $n\not=m$, by Lemma \ref{lemma2} and the definition of the metric $d$, there exists $m_0 \in \N \setminus \{n_0\}$ such that
\begin{itemize}
\item[$(i)$] $\displaystyle \max\{|f_{m_0}(x_{n_0})|,|f_{m_0}(y_{n_0})|\} \leq \frac{\e_{n_0}}{3}$ and
\item [$(ii)$] $\displaystyle \max\{d(x_{n_0}, x_{m_0}), d(x_{n_0}, y_{m_0}), d(y_{n_0}, x_{m_0}), d(y_{n_0}, y_{m_0})\} \leq 1 + \frac{\e_{n_0}}{3}$. 
\end{itemize}

Now, by Lemma \ref{lemma1}, there is a constant $C_{m_0} \in \R$ such that $f_{n_0}(x_{m_0}) = f_{n_0}(y_{m_0})= C_{m_0}$. By relabeling the pairs $(x_{n_0},y_{n_0})$ and $(x_{m_0},y_{m_0})$ if necessary, we may assume that 
$$ |f_{n_0}(x_{n_0})-C_{m_0}|\geq |f_{n_0}(y_{n_0})-C_{m_0}|\qquad\text{and}\qquad|f_{m_0}(x_{m_0})|\geq |f_{m_0}(y_{m_0})|.$$ 

With this assumption, Lemma \ref{fact} yields that 
\begin{equation} \label{ineq3}
|f_{m_0}(x_{m_0})| \geq \frac{1}{2} + \e_{m_0} \qquad\text{and}\qquad |f_{n_0}(x_{n_0})-C_{m_0}| \geq \frac{1}{2} + \e_{n_0}.
\end{equation}
In particular, $f_{m_0}(x_{m_0}) \not= 0$. Set now $\delta := \frac{|f_{m_0}(x_{m_0})|}{f_{m_0}(x_{m_0})}\in\{-1,1\}$. To finish the proof of this case, we distinguish two possibilities according to the sign of $|f_{n_0}(x_{n_0})-C_{m_0}|$:

If $f_{n_0}(x_{n_0}) < C_{m_0}$, consider the function $f=f_{n_0}+\delta f_{m_0}$, which is $1$-Lipschitz by assumption. However, using properties $(i)$ and $(ii)$ and equation \eqref{ineq3} we obtain that
\begin{align*}
|f(x_{n_0}) - f(x_{m_0})| &\geq - f_{n_0}(x_{n_0}) - \delta f_{m_0}(x_{n_0}) + f_{n_0}(x_{m_0}) + \delta f_{m_0}(x_{m_0}) \\
&> \frac{1}{2} + \e_{n_0} + \frac{1}{2} - |f_{m_0}(x_{n_0})|>d(x_{n_0}, x_{m_0}), 
\end{align*} 
a contradiction.

On the other hand, if $f_{n_0}(x_{n_0}) \geq C_{m_0}$, an analogous procedure shows that the function $g=f_{n_0}-\delta f_{m_0}$ has a Lipschitz norm greater than $1$ witnessed by the same pair $(x_{n_0},x_{m_0})$. This again yields a contradiction.

\vspace{0.2cm} 
\noindent
\textit{Case 2}: $C \in\{ B_1,B_2,B_3\}$. 
\vspace{0.2cm} 

We will prove it for $C=B_1$, since the two remaining possibilites can be reduced to this one. Indeed, it is straightforward to check that, if $C=B_2$, by relabelling the pairs $(x_n,y_n)$, for $n\in\N$, we may assume that $C=B_1$. Else, if Ramsey applied to $B_3$, fix $n_0\in S$. For each $m\in S\setminus\{n_0\}$, we have that either $x_{n_0}=y_m$ or $x_m=y_{n_0}$, so one of the sets $S_1=\{m\in S\setminus\{n_0\}\colon x_m=y_{n_0}\}$ or $S_2=\{m\in S\setminus\{n_0\}\colon y_m=x_{n_0}\}$ is infinite, and we have now reduced this situation to the case where Ramsey applies to $B_1$ or $B_2$, respectively.

Hence, by taking subsequences if necessary, we assume that there exists $k^*\in \N$ such that $x_n = p_{k*}$ for every $n \in \N$. By Lemma \ref{lemma1}, we have that $y_n \not= y_m$ for every $n, m \in \N$ with $n \not= m$. Using Lemma \ref{lemma2}, we may find $n_0, m_0 \in \N$ with $n_0 \not= m_0$ such that 
\begin{equation} \label{ineq4} 
|f_{n_0}(x)| \leq \frac{1}{10} \qquad\text{and}\qquad |f_{m_0}(x)| \leq \frac{1}{10}. 
\end{equation}
Hence, applying now Lemma \ref{fact}, we have that 
\begin{equation} \label{ineq5}
|f_{n_0}(y_{n_0})| \geq \frac{9}{10} \qquad\text{and}\qquad |f_{m_0}(y_{m_0})| \geq \frac{9}{10}. 
\end{equation}
Changing signs of $f_{n_0}$ and $f_{m_0}$ if necessary, we may assume that $f_{n_0}(y_{n_0}) > 0$ and $f_{m_0}(y_{m_0}) > 0$. Finally, consider the function $f:=f_{n_0}-f_{m_0}$, which is $1$-Lipschitz by the assumption on the sequence $\{f_n\}_{n\in\N}$. However, applying \eqref{ineq4} and \eqref{ineq5}, and recalling that the diameter of $M$ is $3/2$ we obtain that
\begin{align*}
|f(y_{n_0}) - f(y_{m_0})| &\geq f_{n_0}(y_{n_0}) - f_{m_0}(y_{m_0}) + f_{m_0}(y_{n_0}) - f_{n_0}(y_{m_0}) \\
&\geq \frac{9}{5} - (|f_{m_0}(y_{n_0})| + |f_{n_0}(y_{m_0})|)> d(y_{n_0}, y_{m_0}).
\end{align*} 
This is a contradiction and the proof is over.
\end{proof}

It is worth mentioning that there exist countable bounded uniformly discrete complete metric spaces $M$ with the condition that no point $x$ in $M$ attains its separation radius, but such that $c_0$ embeds isometrically in $\sna(M)$. Indeed, it suffices to consider a countable collection $\{M_n\}_{n\in\N}$ of copies of the previous space in such a way that $d(M_n, M_m)=3$ for all different $n,m\in\mathbb{N}$, and observe that, in this context, Lemma \ref{prop:tent-c02} applies. This means that the aforementioned property is not sufficient for $c_0$ not to be contained in $\sna(M)$ isometrically.

Likewise, one could be tempted to assume that the condition of not attaining the separation radii is at least necessary in negative results as Theorem \ref{main-theorem}. However, this is far from being true as well. In fact, later in this section we will exhibit a proper but not bounded uniformly discrete complete metric space $M$ such that $c_0$ cannot be embedded in $\sna(M)$ isometrically (see Theorem \ref{ContraejemploProper}). In particular, every point of $M$ attains its separation radius, since closed bounded sets in $M$ are compact. On the other hand, we will see in the next subsection that the property of being uniformly discrete is indeed necessary in order to get such negative results (see Theorem \ref{caso-compacto}).

\subsection{Non uniformly discrete metric spaces}

Let us move on to the main positive result of the paper. Going in the opposite direction of Theorem \ref{main-theorem}, here we show that we can \textit{always} embed $c_0$ isometrically in $\sna(M)$ whenever $M$ is infinite but not uniformly discrete.

\begin{theorem} \label{caso-compacto} Let $M$ be an infinite non uniformly discrete metric space. Then, the set $\sna(M)$ contains an isometric copy of $c_0$. 
\end{theorem}

\begin{proof}
By \cite[Theorems 3.2 and 3.4]{AMRTpre} it suffices to assume that $M'$ is non-empty and finite. We also assume without loss of generality that $0\in M'$. Now we can find a sequence $\{x_n\}_{n\in \N}$ in $M$ converging to $0$ such that $R(x_n)>0$ for all $n\in N$. It is clear that the sequence $\{R(x_n)\}_{n\in \N}$ converges to $0$.

We are going to define a sequence $\{f_k\}_{k\in\N}$ of $1$-Lipschitz functions in $\text{SNA}(M)$ which will be isometrically equivalent to the canonical basis of $c_0$ and such that the subspace $\overline{\spn} \{ f_k: k \in \N\}$ (which is isometric to $c_0$) is contained in $\sna(M)$. 

We define the following sets:
\begin{align*}
A &:= \big\{ \{n,m\} \in \N^{[2]}: d(x_n, x_m) \geq R(x_n)+R(x_m) \big\},\\
B &:= \big\{ \{n,m\} \in \N^{[2]}: d(x_n, x_m) < R(x_n)+R(x_m) \big\},
\end{align*}
which form a partition of $\mathbb{N}^{[2]}$. By Ramsey's theorem, there is $C \in \{A, B\}$ and an infinite subset $S \subseteq \N$ such that $S^{[2]} \subseteq C$. These two possibilities give us two separate cases.

\vspace{0.2cm}
\noindent
{\it Case 1}: $C=A$.
\vspace{0.2cm}

Consider the subset $\{x_n\}_{n\in S}$, which satisfies that $d(x_n,x_m)\geq R(x_n)+R(x_m)$ for all $n\neq m\in S$. Assume first that there is an infinite subset of $S$, which we denote by $S$ again, such that $R(x_n)$ is attained for every $n\in S$. Consider now for each $n\in S$ an element $y_n\in M$ such that $d(x_n,y_n)=R(x_n)$. It is straightforward to see that the sequences $\{x_n\}_{n\in S},\{y_n\}_{n\in S}$ satisfy the assumptions of Lemma \ref{prop:tent-c02} and we are done.

Otherwise, we may assume that $R(x_n)$ is not attained for any $n\in S$. Let us then choose inductively a sequence $\{a_k\}_{k\in\N}$ among the elements of the sequence $\{x_n\}_{n\in S}$ satisfying that for every $k\in\N$,
\begin{equation}\label{eqtag1}d(a_k,0)\leq \frac{d(a_j,0)-R(a_j)}{4}\quad \forall j<k.\end{equation}
For the sake of clarity, let us denote $\Delta_k=\frac{d(a_k,0)-R(a_k)}{4}$ for every $k\in\N$. It is clear from \eqref{eqtag1} that $\{\Delta_k\}_{k=1}^\infty$ is a decreasing sequence. Now,
from the fact that $R(a_k)$ is not attained, we deduce that for every $k\in\N$, there is $b_k\in B(a_k,R(a_k)+\Delta_k)$.
Finally, let us prove that the sequences $\{a_k\},\{b_k\}$ are under the assumptions of Lemma \ref{prop:tent-c02}. Pick $n,m\in\N$ with $n<m$. Clearly, the following expressions hold 
\begin{equation}\label{eqtag2}
\begin{aligned}
    d(a_n,0)=&R(a_n)+4\Delta_n,\quad d(a_m,0)\leq\Delta_n,\quad d(a_n,b_n)\leq R(a_n)+\Delta_n\\
    &d(a_m,b_m)\leq R(a_m)+\Delta_m\leq d(a_m,0)+\Delta_m\leq 2\Delta_n.
    \end{aligned}
\end{equation}

Hence, by \eqref{eqtag2} we have that
$$ d(a_n,a_m)\geq d(a_n,0)-d(a_m,0)\geq R(a_n)+3\Delta_n\ge d(a_n,b_n)+d(a_m,b_m). $$
This finishes the first case.

\vspace{0.3cm} 
\noindent
 {\it Case 2}: $C = B$.
\vspace{0.3cm}

Since the set $S$ is infinite and the sequence $\{x_n\}_{n\in \N}$ is convergent, we can inductively define a pair of sequences $\{a_k\}_{k\in\N}\subseteq \{x_n\}_{n\in S}$ and $\{b_k\}_{k\in\N}\subseteq \{x_n\}_{n\in S}$ satisfying the following properties:

\vspace{0.2cm} 
\begin{itemize}
\item[$(i)$] $R(a_k)< \varepsilon_j/2$, for $j,k\in\N$ with $j<k$, where $\varepsilon_j= R(a_j)+R(b_j)-d(a_j,b_j)>0$.
\item[$(ii)$] $R(b_k)<R(a_k)/2$ for every $k\in\mathbb{N}$.
\end{itemize}
\vspace{0.2cm} 

Fixed $k \in \N$, we define $f_k: M \rightarrow \R$ by 
\begin{equation*}
f_k(p) =
\begin{cases}
\displaystyle R(a_k) - \frac{\e_k}{2} & \mbox{if} \ p = a_k, \\
\displaystyle - R(b_k) + \frac{\e_k}{2} & \mbox{if} \ p=b_k, \\
0 & \mbox{otherwise}.
\end{cases}
\end{equation*} 

Property $(i)$ and the definition of $\varepsilon_k$ ensure that $f_k(a_k)\geq 0$ and $f_k(b_k)\leq 0$ for every $k\in\mathbb{N}$. With this, we obtain that
\begin{equation} \label{eq6}
|f_k(a_k)| = R(a_k)- \frac{\e_k}{2},\qquad\text{and}\qquad|f_k(b_k)| = R(b_k)- \frac{\e_k}{2},\qquad\text{for all }k\in \N.
\end{equation}

Let $\{\lambda_k\}_{k \in \N} \in c_0$. Again we will show that $f:= \sum_{k \in \N} \lambda_k f_k \in \sna(M)$ and also that $\|f\|_{\Lip} = \max_{k \in \N} \{ |\lambda_k| \}$. Choose $k_0\in\mathbb{N}$ such that $|\lambda_{k_0}|=\max_{k\in \N}\{|\lambda_k\}$. 

We again start by proving that $f$ is $|\lambda_{k_0}|$-Lipschitz. Take $p, q \in M$ with $p\not=q$. We show that $|f(p)-f(q)|\leq |\lambda_{k_0}|d(p,q)$. If both $p$ and $q$ form a pair $\{a_k,b_k\}$ for some $k\in\mathbb{N}$, the previous inequality is clear. We need to study now the two remaining possibilities:

\vspace{0.3cm} 

\begin{itemize}
    
\item[(a)] Suppose that there exist $k_1, k_2 \in \N$ with $k_1<k_2$ such that $p \in \{ a_{k_1},b_{k_1}\}$ and $q \in \{ a_{k_2}, b_{k_2}\}$. Then, by \eqref{eq6} we have in particular that $|f(p)|=|\lambda_{k_1}|\left( R(p)-\frac{\e_{k_1}}{2}\right)$ and $|f(q)|< |\lambda_{k_2}|R(a_{k_2})$. Hence, we obtain that
\begin{align*}
|f(p) - f(q)| <& |\lambda_{k_0}| \cdot \left( R(p) - \frac{\e_{k_1}}{2} + R(a_{k_2}) \right) \\
\leq & |\lambda_{k_0}| \cdot R(p)\leq |\lambda_{k_0}|\cdot d(p,q).
\end{align*}

\item[(b)] If $p \in M \setminus \{x_k\}_{k \in \N}$, then $f(p) = 0$ and, using (\ref{eq6}) again, we have that 
    \begin{equation*}
    |f(p) - f(q)| = |f(q)| < |\lambda_{k_0}| \cdot R(q) \leq | \lambda_{k_0}| \cdot d(p,q).
    \end{equation*}
\end{itemize}

We have proven then that the Lipschitz norm of $f$ is smaller or equal than $|\lambda_{k_0}|$. Finally, considering the pair of points $a_{k_0}$ and $b_{k_0}$, we quickly observe that $\|f\|_{\Lip}=|\lambda_{k_0}|$ and that $f$ strongly attains its Lipschitz norm at this pair of points. This finishes the proof.
\end{proof}

In \cite[Theorem 3.3]{AMRTpre}, $c_0$ is isomorphically embedded into $\sna(M)$ for countable compact metric spaces $M$ in a non constructive way. Indeed, the authors show that the little Lipschitz space is an infinite-dimensional subspace of $c_0$ contained in $\sna(M)$. The following corollary is an immediate consequence of Theorem \ref{caso-compacto} and improves the aforementioned result from \cite{AMRTpre} by means of a different approach.

\begin{corollary}\label{realcompactcase}
Let $M$ be an infinite compact metric space. Then, the subset $\sna(M)$ contains an isometric copy of $c_0$.
\end{corollary}

\subsection{A proper and uniformly discrete counterexample}

To finish this section, we show that Corollary \ref{realcompactcase} cannot be improved to include proper metric spaces. Indeed, we have the following result. 

\begin{theorem}\label{ContraejemploProper}
There exists an infinite proper uniformly discrete complete metric space $M$ such that $c_0$ is not isometrically contained in $\sna(M)$ and for which every point in $M$ attains its separation radius.
\end{theorem}

\begin{proof}
Let $M=\{p_k\}_{k=0}^{\infty}$, with distinguished point $p_0=0$, be a countable set endowed with the metric $d\colon M\times M\rightarrow \mathbb{R}$ given by:

$$ 
d(p_k,p_j)=
\begin{cases}
    k+j-\varepsilon_{\max\{k,j\}}\quad&\text{if }k\neq j\in \N\setminus\{0\},\\
    k&\text{if }j=0,\\
    j&\text{if }k=0,\\
    0&\text{if }j=k,
\end{cases}
$$
where $\{\varepsilon_k\}_{k\in\N}$ is a sequence of positive numbers such that $\varepsilon_{k+1}>\varepsilon_k$ and $\varepsilon_{k}<1/2$ for all $k\in\mathbb{N}$. For convenience, write $\delta_k=\varepsilon_{k+1}-\varepsilon_{k}>0$ for all $k\in\mathbb{N}$. It is clear that $M$ is proper since every bounded set is finite.

As in the proof of Theorem \ref{main-theorem}, we start by assuming that there exists a sequence $\{f_n\}_{n\in\N}$ of functions in $\sna(M)$ isometrically equivalent to the canonical basis of $c_0$, and we are going to find two natural numbers $n_0\neq m_0$ such that $f_{n_0}- f_{m_0}$ is not $1$-Lipschitz, which will yield a contradiction. For each $n\in\N$, since $f_n$ is strongly norm-attaining, we may consider two points $x_n\neq y_n$ such that $|f(x_n)-f(y_n)|=d(x_n,y_n)$. 

We write $k(n)$ and $j(n)$ to denote the natural numbers such that $x_n=p_{k(n)}$ and $y_n=p_{j(n)}$ for every $n\in \N$. By relabelling the pair $(x_n,y_n)$, we may assume that $k(n)<j(n)$ for all $n\in\mathbb{N}$.

We now define the sets $A$, $B_1$, $B_2$, and $B_3$ as in the proof of Theorem \ref{main-theorem}. By Ramsey's theorem, there exists $C\in\{A,B_1,B_2,B_3\}$ and an infinite set $S\subseteq \mathbb{N}$ such that $S^{[2]}\subseteq C$. Note, however, that the case $C=B_3$ can be reduced to $C\subseteq\{B_1,B_2\}$ as in Theorem \ref{main-theorem}, and the case $C=B_2$ cannot happen, since in that case we would forcefully get functions $f_n$ and $f_m$ from the basis that would strongly attain their norms at the same pair of points, contradicting Lemma \ref{lemma1}. Hence, the conclusion of Ramsey's theorem can only apply to sets $A$ and $B_1$ in this scenario, and thus, by passing to a subsequence if needed, this allows us to reduce the possibilities to only two cases.

% Following the reasoning of the proof of Theorem \ref{main-theorem} we can apply Ramsey's theorem to reduce the problem to two cases:
% \OR{Esto es sólo mi opinión, pero creo que conviene remarcar por qué sólo hay 2 casos, aunque sea sin mucho detalle, porque no es como en el otro teorema que relabeleamos y ya.\\ Creo que basta con poner un comentario del estilo de que ``\color{blue}If the conclusion of Ramsey's theorem applied to the sets $B_2$ or $B_3$, then we would forcefully get functions $f_n$ and $f_m$ of the basis that would attain their norms strongly at the same pair or points, contradicting Lemma \ref{lemma1}. Hence, the conclusion of Ramsey's theorem can only apply to sets $A$ and $B_1$ in this scenario, leading to 2 possible cases.\color{black}''}

\vspace{0.2cm} 
\noindent
\textit{Case 1}: For every $n\neq m$ we have $\{x_n,y_n\}\cap \{x_m,y_m\}=\emptyset$. 
\vspace{0.2cm} 

In this case, choose an arbitrary $n_0\in \N$ such that $k(n_0),j(n_0)\neq 0$. By Lemma \ref{lemma2}, and using that $\{x_n,y_n\}\cap \{x_m,y_m\}=\emptyset$ for all $n\neq m\in\N$, we can find $m_0\in \N$ with $k(m_0)>j(n_0)$ such that 
\begin{equation} \label{ineq8}
    |f_{m_0}(x_{n_0})|<\frac{1}{2}\delta_{j(n_0)}\qquad\text{ and } \qquad|f_{m_0}(y_{n_0})|<\frac{1}{2}\delta_{j(n_0)}.
\end{equation}

Using Lemma \ref{lemma1} we can define $C_{m_0}\in \R$ such that $C_{m_0}=f_{n_0}(x_{m_0})=f_{n_0}(y_{m_0})$. With Lemma \ref{fact} we obtain that either 
\begin{itemize}
    \item[($a_0$)] $|f_{n_0}(x_{n_0})-C_{m_0}|\geq k(n_0)-\frac{1}{2}\varepsilon_{j(n_0)}$, or 
    \item[($a_1$)] $|f_{n_0}(y_{n_0})-C_{m_0}|\geq j(n_0)-\frac{1}{2}\varepsilon_{j(n_0)}$.
\end{itemize}
Similarly and by the same lemma, we have that either 
\begin{itemize}
    \item[($b_0$)] $|f_{m_0}(x_{m_0})|\geq k(m_0)-\big(\frac{1}{2}\varepsilon_{j(n_0)}+\frac{1}{2}\delta_{j(n_0)}\big)$, or 
    \item[($b_1$)] $|f_{m_0}(y_{m_0})|\geq j(m_0)-\big(\varepsilon_{j(m_0)}-\frac{1}{2}\varepsilon_{j(n_0)}-\frac{1}{2}\delta_{j(n_0)}\big)$.
\end{itemize}

In total, there are now 4 different possibilities that must be checked for contradiction. We will only expand on the two possibilities where $(a_0)$ holds, since the two remaining possibilities (when $(a_1)$ holds) are proven similarly. Hence, suppose first that $(a_0)$ and $(b_0)$ hold. By changing the signs of $f_{n_0}$ and $f_{m_0}$ if necessary, we may suppose that $f_{n_0}(x_{n_0})-C_{m_0}\geq k(n_0)-\frac{1}{2}\varepsilon_{j(n_0)}$ and $f_{m_0}(x_{m_0})\geq k(m_0)-\big(\frac{1}{2}\varepsilon_{j(n_0)}+\frac{1}{2}\delta_{j(n_0)}\big)$. Consider the function $f= f_{n_0}-f_{m_0}$, which is $1$-Lipschitz since we are assuming that $\{f_n\}_{n\in\N}$ is isometrically equivalent to the canonical basis of $c_0$. However, using \eqref{ineq8}, we have that
\begin{align*}
    |f(x_{n_0})-f(x_{m_0})|&\geq f_{n_0}(x_{n_0})-f_{m_0}(x_{n_0})-C_{m_0}+f_{m_0}(x_{m_0})\\
    &> k(n_0)-\frac{1}{2}\varepsilon_{j(n_0)}-\frac{1}{2}\delta_{j(n_0)}+k(m_0)-\frac{1}{2}\varepsilon_{j(n_0)}-\frac{1}{2}\delta_{j(n_0)}\\
    &\geq k(n_0)+k(m_0)-\varepsilon_{k(m_0)}= d(x_{n_0},x_{m_0}),
\end{align*}
which yields a contradiction. Suppose now that $(a_0)$ and $(b_1)$ hold. Again we may suppose that $f_{n_0}(x_{n_0})-C_{m_0}\geq k(n_0)-\frac{1}{2}\varepsilon_{j(n_0)}$ and $f_{m_0}(y_{m_0})\geq j(m_0)-\big(\varepsilon_{j(m_0)}-\frac{1}{2}\varepsilon_{j(n_0)}-\frac{1}{2}\delta_{j(n_0)}\big)$. Using \eqref{ineq8} again, the $1$-Lipschitz function $f=f_{n_0}-f_{m_0}$ now tested at the pair $(x_{n_0},y_{m_0})$ yields
\begin{align*}
    |f(x_{n_0})-f(y_{m_0})|&> k(n_0)-\frac{1}{2}\varepsilon_{j(n_0)}-\frac{1}{2}\delta_{j(n_0)}+j(m_0)-\varepsilon_{j(m_0)}+\frac{1}{2}\varepsilon_{j(n_0)}+\frac{1}{2}\delta_{j(n_0)}\\
    &=k(n_0)+j(m_0)-\varepsilon_{j(m_0)}=d(x_{n_0},y_{m_0}),
\end{align*}
which is again a contradiction. This finishes the proof for Case 1.

\vspace{0.2cm} 
\noindent
\textit{Case 2}: $x_n = x_m$ for all $n,m\in \N$. 
\vspace{0.2cm} 
 
Write $k^*$ to denote the natural number (including 0) such that $p_{k^*}=x_n$ for all $n\in \N$. Suppose first that $k^*=0$. Then, choose any two different numbers $n_0\neq m_0\in \N$. Since both $f_{n_0}$ and $f_{m_0}$ strongly attain their norm at the pair $(0,y_{n_0})$ and $(0,y_{m_0})$ respectively, and both $f_{n_0}$ and $f_{m_0}$ vanish at $0$, we have that $|f_{n_0}(y_{n_0})|=j(n_0)$ and $|f_{m_0}(y_{m_0})|=j(m_0)$. With Lemma \ref{lemma1} we obtain that $f_{n_0}(y_{m_0})=f_{m_0}(y_{n_0})=0$. By changing the signs of both functions if needed, we may suppose that $f_{n_0}(y_{n_0})=j(n_0)$ and $f_{m_0}(y_{m_0})=j(m_0)$, producing a contradiction directly by considering the mapping $f=f_{n_0}-f_{m_0}$, which is not $1$-Lipschitz as witnessed by the pair $(y_{n_0},y_{m_0})$. Indeed,
\begin{equation*}
|f(y_{n_0}) - f(y_{m_0})| = j(n_0) + j(m_0) > d(y_{n_0}, y_{m_0}).
\end{equation*}

Suppose now that $k^*\neq 0$. Using Lemma \ref{lemma2}, choose two different natural numbers $n_0\neq m_0\in \N$ with $j(m_0)>j(n_0)>k^*$ such that 
\begin{equation*}
    |f_{n_0}(p_{k^*})|<\frac{1}{4}\qquad\text{ and } \qquad|f_{m_0}(p_{k^*})|<\frac{1}{4}.
\end{equation*}
On the one hand, this means that $|f_{n_0}(y_{n_0})|>k^*+j(n_0)-\varepsilon_{j(n_0)}-\frac{1}{4}$ and $|f_{m_0}(y_{m_0})|>k^*+j(m_0)-\varepsilon_{j(m_0)}-\frac{1}{4}$, while, on the other hand, it implies by Lemma \ref{lemma1} that 
\begin{equation*}
    |f_{n_0}(y_{m_0})|<\frac{1}{4}\qquad\text{ and } \qquad|f_{m_0}(y_{n_0})|<\frac{1}{4}.
\end{equation*}

Finally, we may again suppose without loss of generality that $f_{n_0}$ and $f_{m_0}$ are both positive at the points $y_{n_0}$ and $y_{m_0}$ respectively, and consider the function $f = f_{n_0}-f_{m_0}$, which is assumed to be $1$-Lipschitz. However, we have that
\begin{align*}
    |f(y_{n_0})-f(y_{m_0})|&\geq f_{n_0}(y_{n_0})-f_{m_0}(y_{n_0})-f_{n_0}(y_{m_0})+f_{m_0}(y_{m_0})\\
    &\geq j(n_0)+j(m_0)+2k^*-1-\varepsilon_{j(n_0)}-\varepsilon_{j(m_0)}\\
    &> j(n_0)+j(m_0)>d(y_{n_0},y_{m_0}),
\end{align*}
a contradiction. This finishes the proof of Case 2 and so the theorem is finally proven.
\end{proof}

\section{The non-separable case}

In this section we tackle the problem of embedding $c_0(\Gamma)$ in $\sna(M)$ isometrically, where $\Gamma$ is an arbitrary set of large cardinality. Let us first introduce some basic concepts and results of set theory that will be heavily used in this section.

We denote by $\dens(M)$ the \textit{density character} of a metric space $M$, defined as the smallest cardinal $\Gamma$ such that there is a dense subset of $M$ of cardinality $\Gamma$. The \textit{cofinality} $\cof(\alpha)$ of an ordinal $\alpha$ is the smallest ordinal $\beta$ such that $\alpha = \sup_{\gamma < \beta} \alpha_{\gamma}$, where $\{\alpha_{\gamma}\}_{\gamma<\beta}$ is an ordinal sequence of length $\beta$ with $\alpha_\gamma<\alpha$ for all $\gamma<\beta$. A cardinal $\Gamma$ is \textit{regular} if $\cof(\Gamma)=\Gamma$. Following the notation of \cite{Kunen}, for an ordinal $\alpha$, we denote by $\alpha^+$ the least cardinal strictly bigger than $\alpha$, which is always a regular cardinal (see \cite[Lemma 10.37]{Kunen}). We again refer the reader to \cite{Kunen} for a comprehensive background on this topic. Finally, recall that a subset of a metric space $S\subseteq M$ is called \textit{$r$-separated} for some $r>0$ whenever $d(x,y)\geq r$ for all $x\neq y\in S$.

The next result is essentially based on the proof of \cite[Proposition 3]{HN}.

\begin{proposition} \label{Petr} Let $M$ be a metric space with $\dens(M)=\Gamma > \omega_0$. Then, there exists a discrete set $L \subseteq M$ with $\card(L) = \Gamma$. Moreover, if $\cof(\Gamma) > \omega_0$, then $L$ can be chosen to be uniformly discrete. 
\end{proposition}

\begin{proof} For every $k \in \mathbb{N}$, let $M_k$ be some maximal $\frac{1}{2^k}$-separated subset of $M$. Denote $\Gamma_k := |M_k|$ for all $k \in \mathbb{N}$. If $\cof(\Gamma) > \omega_0$, then, since $\overline{\bigcup_{k \in \mathbb{N}} M_k} = M$, we have that there is $k_0 \in \mathbb{N}$ such that $\Gamma_{k_0} = \Gamma$ and so we take $L := M_{k_0}$.

% For every $k \in \mathbb{N}$, let $\{M_k\}_{k\in\N}$ be a sequence of $\frac{1}{2^k}$-separated subsets of $M$ such that $\bigcup_{k\in\N}M_k$ is dense in $M$. Denote by $\Gamma_k$ the cardinality of the set $M_k$ for all $k \in \mathbb{N}$. If $\Gamma$ has uncountable cofinality, then, since $\bigcup_{k \in \mathbb{N}} M_k$ is dense in $M$, we have that there is $k_0 \in \mathbb{N}$ such that $\Gamma_{k_0} = \Gamma$ and so we take $L := M_{k_0}$.

Now, let us assume that $\cof(\Gamma) = \omega_0$. If there exists $k_0 \in \mathbb{N}$ such that $\Gamma_{k_0} = \Gamma$, we are done, since we can take once again $L:= M_{k_0}$. On the other hand, if this is not the case, we have that $\Gamma_k < \Gamma$ for every $k \in \mathbb{N}$ and $\cof(\Gamma) = \omega_0$. Since $\Gamma$ is not regular, we know that $\Gamma_k^+ < \Gamma$, for every $k \in \mathbb{N}$. Using this, and the fact that $\sup_{k\in\N}\Gamma_k = \Gamma$, it is straightforward to inductively construct a subsequence $\{\Gamma_{k_n}\}_{n\in\mathbb{N}}$ of $\{\Gamma_k\}_{k\in\mathbb{N}}$ with $\Gamma_{k_1}$ infinite and such that $\Gamma_{k_n}^+ < \Gamma_{k_{n+1}}$ for all $n\in\mathbb{N}$.

Now, for each $n\in\mathbb{N}$, let us consider a sequence of sets $\{\widetilde{M}_n\}_{n\in\N}$ such that $\widetilde{M}_n$ is a subset of $M_{k_{n+1}}$ with $|\widetilde{M}_n| = \Gamma_{k_n}^+$ for all $n\in\N$. Let us write $\widetilde{M}_n = \{x_{\alpha}^n: \alpha \in \Gamma_{k_n}^+\}$. 

For each $n\in\mathbb{N}$, each $j\leq n$, and each $\alpha\in \Gamma_{k_j}^+$, we define 
\begin{equation*}
A_{j,\, \alpha}^n := \widetilde{M}_{n+1} \cap B \left( x_{\alpha}^j,\, \frac{1}{2^{k_{j+1} + 1}} \right). 
\end{equation*}

We will inductively construct, for every $n \in \mathbb{N}$, a set $L_n \subseteq \widetilde{M}_n$ with $|L_n| = \Gamma_{k_n}^+$ and a finite subset $N_n\subseteq M$ such that whenever $j<n$, 
\begin{equation} \label{dist-de-los-Ls}
d(L_n,\, L_j \setminus N_n) \geq \frac{1}{2^{{k_{j+1}}+2}}. 
\end{equation}
Set $L_1:= \widetilde{M}_1$ and $N_1:=\emptyset$. Now, assuming that for some $n\in\mathbb{N}$ we have constructed $L_j$ and $N_j$ for all $j\leq n$, we can do the inductive step towards $n+1$.

\begin{itemize}

\item[(a)] Suppose that $|A_{j,\, \alpha}^n| < \Gamma_{k_{n+1}}^+$ for every $j\leq n$ and $\alpha \in \Gamma_{k_n}^+$. Since  $\Gamma_{k_n}^+ < \Gamma_{k_{n+1}}^+$ and $\Gamma_{k_{n+1}}^+$ is regular, we have that
\begin{equation*}
\left| \bigcup_{j \leq n,\, \alpha \in \Gamma_{k_j}^+} A_{j,\, \alpha}^n \right| < \Gamma_{k_{n+1}}^+ = \card(\widetilde{M}_{n+1}).
\end{equation*}
Therefore, the set 
$$L_{n+1}:= \widetilde{M}_{n+1} \setminus \bigcup_{ j \leq n,\, \alpha \in \Gamma_{k_j}^+} A_{j,\, \alpha}^n$$
satisfies $|L_{n+1}|=\Gamma_{k_{n+1}}^+$ and (\ref{dist-de-los-Ls}) holds by setting $N_{n+1} = \emptyset$. Indeed, for any $j\in\{1,\dots,n\}$, every point in $L_j$ is of the form $x_{\alpha}^j$ for some $\alpha\in\Gamma_{k_j}^+$. Hence, if there exists a point $p\in L_{n+1}$ such that $d(p,x_{\alpha}^j)<\frac{1}{2^{{k_{j+1}}+2}}$, then $p$ belongs to the set $A_{j,\, \alpha}^n$, which leads to a contradiction with the definition of $L_{n+1}$.

\item[(b)] Suppose now that $|A_{j_0,\, \alpha_0}^n| = \Gamma_{k_{n+1}}^+$ for some $j_0\leq n$ and some $\alpha_0 \in \Gamma_{k_{j_0}}^+$. Without loss of generality we consider $j_0\in\{1,\dots,n\}$ to be such that $|A_{j,\, \alpha}^n|<\Gamma_{k_{n+1}}^+$ for all $j_0<j\leq n$ and all $\alpha\in \Gamma_{k_j}^+$. Define 
$$L_{n+1}:= A_{j_0,\, \alpha_0}^n \setminus \bigcup_{j_0<j\leq n,\, \alpha\in\Gamma_{k_j}^+} A_{j,\, \alpha}^n.$$ 
Arguing as in case $(a)$, we obtain that $|L_{n+1}|=\Gamma_{k_{n+1}}^+$. Finally, define
$$ N_{n+1} := \left\{x\in M\colon \exists i\in\{1,\dots,j_0\}\text{ such that }x\in L_i\text{ and }d(x,L_{n+1})< \frac{1}{2^{{k_{i+1}}+2}}\right\}$$
which is finite since for each $i\in\{1,\dots,j_0\}$, there can only be at most a single point $x_i$ in $L_i$ such that $d(x_i,L_{n+1})< \frac{1}{2^{{k_{i+1}}+2}}$. Indeed, if $i=j_0$, the only point in $L_{j_0}$ that can satisfy that property is $x_{\alpha_0}^{j_0}$, since for every $\beta\in \Gamma_{k_{j_0}}^+\setminus \{\alpha_0\}$, $d(x_{\beta}^{j_0}, A_{j_0,\,\alpha_0}^{n})\geq \frac{1}{2^{k_{j_0 + 1}+1}}$. On the other hand, if $i<j_0$, if there were two points $x_i\neq y_i\in L_i$ with that property, we would have that
\begin{align*}
    d(x_i,y_i)&\leq d(x_i,A_{j_0,\, \alpha_0}^n)+d(y_i,A_{j_0,\, \alpha_0}^n)+\diam(A_{j_0,\, \alpha_0}^n)<\frac{1}{2^{{k_{i+1}}}},
\end{align*}
a contradiction with the fact that $L_i$ is $\frac{1}{2^{k_{i+1}}}$-separated.

Let us check that the sets $L_{n+1}$ and $N_{n+1}$ satisfy equation \eqref{dist-de-los-Ls} for each $j\in\{1,\dots,n\}$. Fix $j\in\{1,\dots,n\}$. If $j\leq j_0$ then the inequality follows directly by definition of $N_{n+1}$. Otherwise, if $j>j_0$, then the inequality holds following the same argument as in case (a).

\end{itemize}

Having discussed both possibilities, the induction is finished. To finish the proof, set $L:= \left( \bigcup_{n\in\mathbb{N}} L_n \right) \setminus \left( \bigcup_{n\in\mathbb{N}} N_n \right)$. It is clear that $|L|=\Gamma$, and using equation \eqref{dist-de-los-Ls}, it is straightforward to prove that all convergent sequences in $L$ are eventually constant, and thus $L$ is discrete. 
\end{proof}

As an application of Lemma \ref{prop:tent-c02} and Proposition \ref{Petr}, we have the following isometric result.

\begin{theorem} \label{gamma}
Let $M$ be a pointed metric space such that $\dens(M')=\Gamma$ for some infinite cardinal $\Gamma$. Then there is a linear subspace of $\sna(M)$ that is isometrically isomorphic to $c_0(\Gamma)$.
\end{theorem}

\begin{proof}
The case where $\Gamma=\omega_0$ is already covered in \cite[Theorem 3.2]{AMRTpre}. Assume now that $\Gamma>\omega_0$. If we apply Proposition \ref{Petr} to the set $M'$, we find a discrete set $L\subseteq M'$ with $\card(L)=\dens(L)=\Gamma$ and such that all points of $L$ are cluster points of $M$. Finally, Lemma \ref{prop:tent-c02} can be applied now if we consider $\{x_\gamma\}_{\gamma\in \Gamma}$ to be $L$ itself and for each $\gamma\in\Gamma$, we set $y_\gamma$ to be sufficiently close to $x_\gamma$.
\end{proof}

\noindent \textbf{Acknowledgments.} The authors would like to thank Petr Hájek for some of his inputs on Proposition \ref{Petr} and Vladimir Kadets for some of his comments on Lemma \ref{prop:tent-c02}. They also want to thank Vicente Montesinos for fruitful conversations on the topic of the paper. This research was partially done when the second, third, and fourth authors were visiting the Universitat Jaume I, and they are thankful for the hospitality that they received.

\noindent 
\textbf{Funding information}: S. Dantas was supported by the Spanish AEI Project PID2019 - 106529GB - I00 / AEI / 10.13039/501100011033 and also by Spanish AEI Project PID2021-122126NB-C33 / MCIN / AEI / 10.13039 / 501100011033 (FEDER). R. Medina was supported by CAAS CZ.02.1.01/0.0/0.0/16-019/0000778, project SGS21 / 056 / OHK3 / 1T / 13, Spanish AEI Project PID2021-122126NB-C31, and MIU (Spain) FPU19/04085 Grant. A. Quilis was supported by PAID-01-19, by CAAS CZ.02.1.01/0.0/0.0/16-019/0000778 and by project SGS21 / 056 / OHK3 / 1T / 13. Ó. Roldán was supported by the Spanish Ministerio de Universidades, grant FPU17/02023, and by projects MTM2017-83262-C2-1-P / MCIN / AEI / 10.13039 / 501100011033 (FEDER) and PID2021-122126NB-C33 / MCIN / AEI / 10.13039 / 501100011033 (FEDER).

%%%%%%%%%%%%%%%%%%%%%%%%

\end{document}